\documentclass[11pt]{amsart}

\usepackage{amsmath,amssymb}
\usepackage{mathtools}

\usepackage{tikz}
\usetikzlibrary{patterns}
\usepackage[ruled,vlined,linesnumbered,norelsize]{algorithm2e}
\SetEndCharOfAlgoLine{}
\SetKwProg{Function}{Function}{}{}
\SetKwInOut{Input}{Input}
\SetKwInOut{Output}{Output}

\newcommand{\lcm}{\mathrm{lcm}}
\newcommand{\Z}{\mathbb{Z}}

\newcommand{\ZZ}{\mathbb{Z}}
\newcommand{\RR}{\mathbb{R}}
\newcommand{\ord}{\mathrm{ord}}

\newcommand{\ORDER}{\mathcal{O}}

\SetKwFunction{Landau}{SquareFree}
\SetKwFunction{SplitLLL}{SplitLLL}
\SetKwFunction{SplitCF}{SplitCF}
\SetKwFunction{Refine}{Refine}
\SetKwFunction{CleanDivisors}{CleanDivisors}
\SetKwFunction{ACDCF}{{ACD\_CF}}
\SetKwFunction{ACDL}{{ACD\_LLL}}

\SetKwFunction{ShortVector}{ShortVector}
\SetKwFunction{BivariateCoppersmith}{BivariateCoppersmith}
\SetKwFunction{FactorizationWithPhiTwo}{FactorizationWithPhi2}
\SetKwFunction{FactoringWithKnownDifference}{FactoringWithKnownDifference}

\SetKwFunction{FactorWithEvenPower}{FactorWithEvenPower}
\SetKwFunction{SplitWithOracleRandom}{SplitWithOracleRandom}
\SetKwFunction{FactorWithOracle}{FactorWithOracle}
\SetKwFunction{FactorWithOracleRandom}{FactorWithOracleRandom}
\SetKwFunction{FWO}{FWO}
\SetKwFunction{FactorWithFactoredOrder}{FactorWithFactoredOrder}
\SetKwFunction{Factorization}{Factorization}

\newcommand{\ACDP}{{ACDP}}

\usepackage{hyperref} 
\usepackage{url}

\makeatletter
\newcommand{\addresseshere}{%
  \enddoc@text\let\enddoc@text\relax
}
\makeatother

\newtheorem{lemma}{Lemma}[section]
\newtheorem{theorem}[lemma]{Theorem}
\newtheorem{proposition}[lemma]{Proposition}
\newtheorem{corollary}[lemma]{Corollary}
\theoremstyle{definition}
\newtheorem{definition}[lemma]{Definition}
\theoremstyle{remark}
\newtheorem{remark}{Remark}[section]
\newtheorem{example}{Example}[subsection]

\begin{document}

\title{Deterministic factoring with oracles}
\author{François  Morain \and Guénaël Renault \and Benjamin Smith}
\address[F.~Morain and B.~Smith]{
   LIX - Laboratoire d'informatique de l'\'Ecole polytechnique\\
   GRACE - Inria Saclay - Ile de France}
\email[F.~Morain]{morain@lix.polytechnique.fr}
\email[B.~Smith]{smith@lix.polytechnique.fr}
\address[G.~Renault]{
Agence Nationale de la S\'ecurit\'e des Syst\`emes d'Information\\
and 
LIX - Laboratoire d'informatique de l'École polytechnique,
    CNRS,
    Institut Polytechnique de Paris
    \emph{and}
    GRACE - Inria Saclay--Île-de-France}
\email[G.~Renault]{guenael.renault@ssi.gouv.fr}

\date{\today}

\maketitle

\begin{abstract}
    Can we factor an integer \(N\) 
    unconditionally, in deterministic polynomial time,
    given the value of its Euler totient \(\varphi(N)\)?
    We show that this can be done
    under certain size conditions on the prime factors of \(N\).
    The key technique is lattice basis reduction using the LLL algorithm.
    Among our results, we show that 
    if \(N\) has a prime factor \(p > \sqrt{N}\),
    then we can recover \(p\) 
    in deterministic polynomial time given \(\varphi(N)\).
    We also shed some light on 
    the analogous factorization problems given oracles for the sum-of-divisors function,
    Carmichael's function, and the order oracle 
    that is used in Shor's quantum factoring algorithm.
\end{abstract}

\section{
    Introduction
}

The \emph{fundamental theorem of arithmetic}
states that every positive integer $N$ can be written in a unique way,
up to permutation of the factors, as
$$N = \prod_{i=1}^k p_i^{e_i}$$
where the $p_i$ are distinct primes, and each $e_i > 0$.
Making this theorem explicit by
computing the prime factorization of \(N\)---that is,
computing the \(p_i\) and \(e_i\)---is a 
fundamental problem in algorithmic number theory.
This article is concerned with \emph{deterministic}
factorization algorithms.

Some numbers are easy to factor deterministically.
If $N$ is prime, then Miller~\cite{Miller75} proved that $N$
can be proven prime in deterministic polynomial time assuming the
Generalized Riemann Hypothesis (see also~\cite{Lenstra79}). The same
result was proven unconditionally in~\cite{AgKaSa04}.
In practice,
small numbers can be proven prime using a combination of
pseudoprimality tests.
For large numbers, several faster (though
heuristic) methods exist: see~\cite{CrPo05} for details.
Prime powers can be detected in quasi-linear time~\cite{BeLePi07}.

But when \(N\) has more than one prime factor, hard work is generally
required. In the quantum world, we can apply Shor's
algorithm~\cite{Shor97}.  In the classical world, the fastest
algorithms are non-deterministic: depending on the size of \(N\),
one may use Lenstra's ECM or the Number Field Sieve (NFS),
the best general-purpose factoring algorithm,
which runs in heuristic time
$\exp((\sqrt[3]{64/9}+o(1)) (\log N)^{1/3} (\log\log N)^{2/3})$ 
\cite{CrPo05}. 
This complexity explains the success of the RSA cryptosystem,
which is based on the supposed difficulty of factoring numbers with
only two prime factors.

Deterministic unconditional factoring methods are rare;
all such methods known have
exponential running time for general \(N\).
The first such method was due to Fermat,
followed by Lehman \cite{Lehman74};
Pollard's approach \cite{Pollard74}
has been built on by
recent methods
including Bostan--Gaudry--Schost~\cite{BoGaSc07},
Costa--Harvey~\cite{CoHa14}, and Hittmeir~\cite{Hittmeir18},
all in time $\tilde{O}(N^{1/4})$.
More recently, this complexity has been improved
to $\tilde{O}(N^{2/9})$ by Hittmeir~\cite{Hittmeir20},
and to $\tilde{O}(N^{1/5})$ by Harvey~\cite{Harvey20}
(see also~\cite{HaHi21} for a later speedup).
Better results exist for numbers known to have special forms:
for example,
\cite{BoDuHo99} describes a method to factor $N = p_1^r p_2$
that runs in polynomial time 
when \(p_1\) and \(p_2\) are of roughly the same size
and \(r\) is in \(\Omega(\log p_1)\).
This was extended in \cite{CoFaReZe16} to 
numbers $N = p_1^r p_2^s$ with $r$ and/or $s$ in $\Omega((\log p_1)^3)$.

The use of \emph{oracles}
allows us to abstract and encapsulate the availability of extra information
about the number \(N\).
It is thus a traditional way 
of trying to understand the difficulty of factoring.
In this work, we consider factoring algorithms
with access to one of the following oracles (defined formally
in~\S\ref{sec:oracle-defs}):
\begin{itemize}
    \item $\Phi$: on input $N$ returns $\varphi(N)$, the value of the
        Euler totient function;
    \item $\Lambda$: on input $N$ returns $\lambda(N)$, 
        the value of the Carmichael lambda function;
    \item $\ORDER$: on input $N$ and $a$ with \(\gcd(a,N) = 1\),
        returns the order of $a$ modulo~$N$;
    \item $\Sigma$: on input $N$ returns $\sigma(N)$, the sum of all
        positive divisors of $N$.
\end{itemize}
We study the conditions under which these oracles can be used to
factor $N$ deterministically,
unconditionally, and in a time complexity better than exponential,
in the spirit of \cite[Rem$23_{86}$]{AdMc94}.

The story of factoring with oracles
began with Miller~\cite{Miller75}, who proved the equivalence of
$\Phi$ and factoring under ERH. Long~\cite{Long81} proved
that factoring is randomly polynomially equivalent to computing orders.
Woll~\cite{Woll87} explored relationships between number-theoretic
problems including factorization and the \(\Phi\) and \(\ORDER\) oracles.
\'Zra\l{}ek~\cite{Zralek10} has shown that almost all integers $N$ can be
factored in deterministic polynomial time given $\varphi(N)$; also,
iterated calls to $\Phi$ allow deterministic factoring in
subexponential time, after using Landau's algorithm to
reduce to the squarefree case (see \S\ref{sct:sqf}).
This work was subsequently extended in~\cite{Zralek19}
(using methods tangential to ours).

In a different direction, Bach, Miller, and Shallit~\cite{BaMiSh86}
showed that $\Sigma$ allows efficient randomized factoring
(see~\S\ref{sec:randomized}).
Chow~\cite{Chow15} has studied factoring
with an oracle of a completely different
nature, using coefficients of modular forms;
this turns out to be very powerful, since it solves the integer
factorization problem.

There is also an important practical motivation for oracles in factoring.
In the context of RSA moduli $N=p_1p_2$, the
problem of factoring given additional information on \(p_1\) and
\(p_2\) has been studied since 1985.  
For example,
Rivest and Shamir showed in \cite{RiSh85}
that if $N$ has bitlength $n$ and the factors \(p_1\)
and \(p_2\) are balanced (with bitlengths close to $\frac{n}{2}$),
then \(N\) can be factored in polynomial time if we have access to an
oracle returning the $\frac{n}{3}$ most significant bits of~$p_1$.
Beyond their theoretical interest, these algorithms are motivated by
cryptographic hardware attacks: the oracle is an abstraction
representing side-channel analysis revealing some of the bits of the
secret factors.  In 1996, Coppersmith improved Rivest and Shamir's
results by applying lattice-based methods to the problem of finding
small integer roots of bivariate integer polynomials (what is now
called \emph{Coppersmith's method}~\cite{Coppersmith96a}). For
instance, knowing the $n/4$ most (or least) significant bits of~$p$
is enough to factor $N$ in polynomial time.
In the same cryptographic context, the
Coppersmith approach was used to prove that given a pair of RSA
exponents $(e, d)$ (with $d \equiv 1/e\bmod \varphi(N)$), one can
recover the two prime factors of $N$ in deterministic polynomial time~\cite{CoMa07}.

In this article we combine these approaches,
applying lattice-based techniques 
to factoring with number-theoretic oracles.
Our results rely on diophantine geometry, using classical continued
fractions and the LLL algorithm in a manner inspired by
the cryptographic work mentioned above.
Our results include the following:
\begin{theorem}\label{thm11}
    Assume $N$ is squarefree and has at least three prime factors, 
    of which the largest~\(p\) satisfies $p > \sqrt{N}$. 
    Then we can recover \(p\) in deterministic polynomial
    time in $\log(N)$ given one of 
    $\varphi(N)$, $\lambda(N)$, or $\sigma(N)$.
\end{theorem}
\begin{proof}
    See Theorem~\ref{propcor:key}.
\end{proof}

\begin{theorem}
    Assume \(N\) is squarefree and has exactly three prime factors
    $p_1 > p_2 > p_3$. Put $\alpha_i = \log p_i/\log N$.
    Then we can compute a nontrivial factor of \(N\)
    in deterministic polynomial time in \(\log(N)\)
    given \(\varphi(N)\) or \(\sigma(N)\)
    if \emph{at least one} of the following conditions hold:
    \begin{enumerate}
        \item
            \(\alpha_1 > 1/2\); or
        \item
            \(2 \alpha_1 + 3 \alpha_2 \geq 2\); or
        \item
            \(\alpha_2 > (-1 + \sqrt{17})/8\).
    \end{enumerate}
\end{theorem}
\begin{proof}
    Follows from Theorems~\ref{thm52},
    \ref{th:ACDL-theoretical},
    \ref{th:ACDL-practical},
    and~\ref{propcor:key}.
\end{proof}

We define the oracles,
and recall some associated number-theoretic results,
in~\S\ref{sec:oracles},
before
stating the relevant results of Coppersmith and Howgrave-Graham
in~\S\ref{sct-tools}.
Our core results in~\S\ref{sct-prep}
solve (generalizations of) the following problem:
given \(N\) and \(M\) such that 
there exists a (large enough) prime~\(p\) with $p \mid N$ and $p\pm 1 \mid M$,
recover~$p$ in deterministic polynomial time.
We apply these algorithms to factoring 
with \(\Phi\), \(\Lambda\), and \(\Sigma\) in~\S\ref{sct-Phi},
and with \(\ORDER\) and other oracles in~\S\ref{sec:other}.

\section{
    Number-theoretic oracles
}
\label{sec:oracles}

As above, suppose
$
    N = \prod_{i=1}^k p_i^{e_i}
$,
where the $p_i$ are distinct primes and $e_i > 0$.
Let \(\omega(N)\) denote the number of prime divisors of~\(N\)
(so \(\omega(N) = k\) above).
Recall that \(\omega(N)\) is trivially bounded above by $(\log N)/(\log 2)$,
and is of order $\log \log N$ on average. 

\subsection{The oracles}
\label{sec:oracle-defs}

\begin{definition}[The \texorpdfstring{\(\Phi\)}{Phi} oracle]
    \label{def:oracle-phi}
    Given \(N\) as above,
    the oracle \(\Phi\) returns the value of
    the Euler totient function
    $$
        \varphi(N) = \prod_{i=1}^{\omega(N)} p_i^{e_i-1} (p_i-1)
        \,,
    $$
    which counts the number of integers in $\{1, \ldots, N-1\}$
    that are prime to $N$;
    that is, \(\varphi(N)\) is the cardinality of the multiplicative group~\((\ZZ/N\ZZ)^\times\).
\end{definition}

\begin{definition}[{The \texorpdfstring{\(\Lambda\)}{Lambda} oracle}]
    \label{def:oracle-lambda}
    Given \(N\) as above,
    the oracle \(\Lambda\) returns the value of 
    Carmichael's \(\lambda\) function
    \[
        \lambda(N) = \lcm_{i=1}^{\omega(N)} \lambda(p_i^{e_i})
        \quad
        \text{where}
        \quad
        \lambda(p_i^{e_i})
        =
        \begin{cases}
            1 & \text{if } p_i = 2 \text{ and } e_i = 1,
            \\
            2 & \text{if } p_i = 2 \text{ and } e_i = 2,
            \\
            \varphi(2^{e_i})/2 & \text{if } p_i = 2 \text{ and } e_i > 2.
            \\
            \varphi(p_i^{e_i}) & \text{if } p_i > 2
            \,.
        \end{cases}
    \]
    This is
    the exponent of \((\ZZ/N\ZZ)^\times\):
    that is, the maximal multiplicative order of an element modulo~\(N\).
\end{definition}

\begin{definition}[The \texorpdfstring{\(\ORDER\)}{order} oracle]
    \label{def:oracle-order}
    Given \(N\) as above and \(a\) with \(\gcd(N,a) = 1\),
    the oracle \(\ORDER\) returns the order
    \[
        \ord_N(a) := \min \{ r : r \in \ZZ_{>0} \mid a^r \equiv 1 \pmod{N} \}
        \,.
    \]
\end{definition}

Shor's quantum factorization algorithm
uses the Quantum Fourier Transform
to construct a quantum polynomial-time order-finding algorithm,
which yields an efficient factorization algorithm
after some classical post-processing 
(similar to the process in \S\ref{sec:randomized} below).
This order-finding algorithm
is not a true realization of \(\ORDER\),
since it is only guaranteed to return a \emph{divisor} of \(\ord_N(a)\),
but for most inputs it returns the true order
with very high probability.
Factoring with \(\ORDER\)
therefore gives us valuable intuition into Shor-style 
quantum factoring algorithms.

\begin{definition}[{The \texorpdfstring{\(\Sigma\)}{sigma} oracle}]
    \label{def:oracle-sigma}
    Given \(N\) as above,
    the oracle \(\Sigma\) returns the sum of the divisors of \(N\):
    that is,
    $$
        \sigma(N) 
        := 
        \sum_{d\mid N} d = \prod_{i=1}^{\omega(N)} \frac{p_i^{e_i+1}-1}{p_i-1}
        \,.
    $$
\end{definition}

\subsection{Relationships between \texorpdfstring{$\Phi$, $\Lambda$, and $\ORDER$}{the oracles}}
\label{sec:relationships}

Lagrange's theorem tells us that the order of an element divides the
order, and indeed the exponent, of the group.
Applying this to \((\ZZ/N\ZZ)^\times\)
gives
\[
    \ord_N(a) \mid \lambda(N)
    \quad
    \text{and}
    \quad
    \lambda(N) \mid \varphi(N)
\]
for all \(N\) and all \(a\) prime to \(N\).

While the \(\varphi\) and \(\lambda\) functions may seem very close, 
it is easy to see that $\varphi(N)/\lambda(N)$ can be made quite large.
For example, if \(N = p_1p_2\) where $p_1-1 = 2(p_2-1)$,
then $\varphi(N)/\lambda(N) = p_2-1 = \Omega(\sqrt{N})$.

Recall that if \(p\) is a prime,
then the valuation \(\nu_p(x)\) of an integer \(x\) at \(p\)
is the maximal~\(e\) such that \(p^e\mid x\).
If \(N\) is odd,
then \(\nu_2(\varphi(N)) = \sum_{i=1}^{\omega(N)}\nu_2(p_i-1) \ge \omega(N)\)
is an easy upper bound for \(\omega(N)\),
which may be useful when we have access to \(\Phi\)
(though this bound is generally far from tight).
In contrast,
\(\nu_2(\lambda(N)) = \max_{i=1}^{\omega(N)}\nu_2(p_i-1)\) 
gives us no information about \(\omega(N)\) 
on its own---and so neither does \(\nu_2(\ord_N(a))\) for any \(a\).

\subsection{Randomized and conditional algorithms}
\label{sec:randomized}

All of these oracles give efficient \emph{randomized} factoring algorithms
(see~\cite{BaMiSh86}). When $N$ is composite, $\varphi(N)$ and
$\lambda(N)$ are even,
which enables us to find some $c \neq \pm 1$ in \(\ZZ/N\ZZ\) 
such that $c^2 \equiv 1 \pmod N$, and then
$\gcd(c-1, N)$ is a nontrivial factor of~\(N\).
See Appendix~\ref{sec:appendix} for the corresponding algorithms. For $\Sigma$, we
refer to \cite{BaMiSh86} again.

    Folklore tells us that there is a randomized polynomial-time reduction
    between computing square roots modulo \(N\) and factoring \(N\).
    Rabin gives a precise analysis when \(N\) is a product of two primes
    in~\cite[Theorem 1]{Rabin79}. 
    To render this approach deterministic
    (as in~\cite{Miller75})
    one needs a bound on non-quadratic residues,
    but this bound is currently only known to hold under ERH.

\section{
    Lattices, Coppersmith's method, and approximate common divisors
}
\label{sct-tools}

In this section 
we recall some essential results
on our two basic tools:
Coppersmith's method for finding small roots of polynomials,
and Howgrave-Graham's approximate common divisors.
We also introduce some elementary subroutines
that we will use to improve the quality of our factorizations.

The Lenstra--Lenstra--Lov\'asz lattice basis reduction algorithm
(LLL)~\cite{LeLeLo82}
is at the heart of both Coppersmith's and Howgrave-Graham's
methods.
Recall that if \(L\) is a lattice of dimension \(n\) in~\(\RR^n\)
(with the Euclidean norm \(\Vert\cdot\Vert\)), then LLL
produces a basis $(b_1, b_2, \ldots, b_n)$ of \(L\) satisfying
(among other conditions)
\[
    \Vert b_1\Vert \leq 2^{(n-1)/4} \mathrm{det}(L)^{1/n}
    \,.
\]
The LLL algorithm computes an LLL-reduced basis for $L$
in polynomial time in \(n\), and in $\log B$ 
where $B$ is a bound on all $\Vert b_i\Vert^2$.
The resulting \(b_1\) 
is approximately as short as possible:
\(\Vert b_1\Vert \le 2^{(n-1)/2}\min_{v\in L\setminus\{0\}}\Vert
v\Vert\). Note that all our lattices will have integer coefficients.

Many variants of LLL have been designed for more speed
and accuracy (e.g. \cite{SchnorrEuchner94,NguyenStehle09,NovocinSV11}),
but the original LLL algorithm suffices for our results.

\subsection{Bivariate Coppersmith}

Theorem~\ref{cop} describes the input and output of Coppersmith's method
for finding small zeroes of integer bivariate
polynomials~\cite{Coppersmith96a,Coppersmith97}.
Coppersmith's algorithm is clarified in~\cite{Coron04} 
and~\cite{BlMa05},
and extended to the general multivariate case 
in~\cite{Jutla98}, \cite{Coppersmith01}, \cite{BlMa05}, 
and~\cite{Ritzenhofen2010}.

\begin{theorem}\label{cop}
    Let $f(x, y) = \sum p_{i, j} x^i y^j\in \Z[x, y]$ be irreducible, 
    of degree at most $\delta$ in \(x\) and \(y\),
    and suppose $f(x_0,y_0) = 0$ 
    for some $|x_0| < X$, $|y_0| < Y$. 
    If 
    \[
        X Y < \mathcal{W}^{2/(3\delta)}
        \quad
        \text{where}
        \quad
        \mathcal{W} = \Vert f(x X, y Y)\Vert^* 
        := \max_{i, j} |p_{i, j}| X^i Y^j
        \ ,
    \]
    then we can find all such solutions $(x_0,y_0)$ 
    in deterministic polynomial time in \(\log \mathcal{W}\) and \(\delta\).
\end{theorem}
\begin{proof}
    See Coron's treatment in~\cite{Coron07}.
\end{proof}

In this article we will apply the special case of
Theorem~\ref{cop} where the polynomial \(f\) is linear in each variable
to find divisors of $N$. 
In another direction, but using the same techniques,
Theorem~\ref{theorem:co-ho-na} improves on a result of
Lenstra~\cite{Lenstra84}.
\begin{theorem}[Coppersmith--Howgrave-Graham--Nagaraj~\cite{CoHoNa08}]
    \label{theorem:co-ho-na}
    Let $0 \leq r < s < N$ with $\gcd(r, s)=1$
    and $s \geq N^{\alpha}$ for some $\alpha > 1/4$.
    The number of divisors of $N$ 
    that are congruent to $r \pmod s$ is in
    $O((\alpha-1/4)^{-3/2})$. The divisors can be found in
    deterministic polynomial time.
\end{theorem}

\subsection{Approximate common divisors}
\label{sct-agcd}

One of the first applications of Coppersmith's method
was to attack RSA moduli,
factoring \(N = p_1p_2\) in polynomial time
given half of the bits of $p_1$.
The algorithmic presentation of these theorems used today
is due to Howgrave-Graham~\cite{Howgrave-Graham97},
who later used this result to solve the \emph{Approximate Common
Divisor Problem} (ACDP)~\cite{Howgrave-Graham01},
which we formalize in Definition~\ref{def:ACDP}.

\begin{definition}[ACDP]
    \label{def:ACDP}
    Given integers $A$ and $B$, and bounds $X$ and $D_0$ 
    for which there exists at least one $(x,D)$ with $|x| \leq X$
    and
    \(D > D_0\) such that \(D \mid B\) and $D \mid (A+x)$,
    the \ACDP{} is to find all such $(x,D)$.
\end{definition}
Before going further, we must make the following very important
observation (not present in~\cite{Howgrave-Graham01}).
\begin{remark}\label{sqf-A-B}
    If $(x,D)$ is an \ACDP{} solution for \((A,B,X,D_0)\),
    then so is $(x,D/z)$ for any divisor $z$ of $D$ 
    such that $D/z > D_0$.
\end{remark}

Howgrave-Graham gives two types of algorithms for solving \ACDP{}
instances in~\cite{Howgrave-Graham01}.
The first, using continued fractions, is described by 
Proposition~\ref{prop:CF} and Algorithm~\ref{alg:ACDCF} (\ACDCF).
The second approach, using LLL, is described by Theorem~\ref{thm:ACD_L}
and Algorithm~\ref{alg:ACDL} (\ACDL). 
Both algorithms run
in deterministic polynomial time, unlike the algorithms for the
\emph{Generalized} Approximate Common Divisor problem (GACDL) 
also considered in~\cite{Howgrave-Graham01}.

As noted in~\cite{Howgrave-Graham01}, 
the continued fraction (\ACDCF) and lattice (\ACDL) approaches 
are subtly different:
\ACDCF requires only a lower bound on one exponent $\alpha$,
but \(\ACDL\) requires some relation between two exponents,
$\alpha$ and $\beta$ (or $\epsilon$).
We will encounter this difference in~\S\ref{varphikgt2}.
Similar phenomena appear in the context of
\emph{implicit factorization}
(e.g. \cite{MaRi09,SaMa09,FaMaRe10,SaMa11}),
but in these cases the two exponents can be handled more easily.

\subsection{Computing approximate common divisors via continued fractions}

The following is taken from \cite{Howgrave-Graham01}.
We include the proof here,
because we will need to be precise about what the algorithm actually
outputs.

\begin{proposition}[Howgrave-Graham]\label{HG_CF}
    Given integers \(A < B\), and real $\alpha > 1/2$, we can find all
    integers $|x_0| < X = B^{2\alpha-1}/2$ such that there exists $D >
    B^{\alpha}$ dividing both $A+x_0$ and $B$, or decide that no such
    $x_0$ exists, in deterministic polynomial time in \(\log B\).
\end{proposition}
\begin{proof}
    Suppose \((x, D)\) is one of the desired \ACDP{} solutions:
    then 
    \(D \mid B\) and \(D \mid (A+x)\),
    with \(D > B^\alpha\) and \(|x| \le B^{2\alpha-1}/2\).
    Write
    $a' = (A+x)/D$ and $b' = B/D$;
    then
    \(
        b' < B^{1-\alpha}
    \),
    from which
    \[
        \left|\frac{A}{B} - \frac{a'}{b'}\right|
        = 
        \frac{|x|}{B} 
        <
        \frac{1}{2 (b')^2}
        \ .
    \]
    The classical theory of continued fraction approximations
    tells us that $a'/b'$ must be one of the convergents
    \((g_1/h_1, g_2/h_2, \ldots)\) of $A/B$ (see for example
    \cite[Theorem~9.10]{LeVeque96}).
    We note that the convergents are
    obtained in reduced form, that is, with $\gcd(g_i, h_i)=1$; 
    the $h_i$ are strictly increasing; and the last term
    is $h_k = B$. Last but not least, this sequence is finite and has
    polynomial size in $\log B$
    (this is closely related to the computation of $\gcd(A, B)$, and
    can be done in deterministic polynomial time~\cite{Knuth97,GathenGerhard}). 

    A solution yields $(A+x)/B = g_i/h_i$: that is, $h_i (A+x) = B
    g_i$. Since $g_i$ and $h_i$ are coprime, this implies that $g_i
    \mid A+x$.
    Put $D = (A+x)/g_i$. Now $h_i D = B$, and $h_i$ must divide~$B$. 
    If this is the case, then we have recovered the \ACDP{} solution 
    $(x,D) = (Dg_i-A, B/h_i)$.
    We can stop as soon as \(h_i \geq B^{1-\alpha}\),
    because such \(h_i\) cannot yield $D > B^{\alpha}$.
\end{proof}
\begin{remark}\label{AltB}
    As noted in \cite{Howgrave-Graham01}, if our problem requires $A > B$,
    then we can replace $A$ by using $q$ such that $A-q B < B$.
\end{remark}

\begin{proposition}\label{prop:CF}
    Given integers \(A < B\),
    Algorithm~\ref{alg:ACDCF} (\ACDCF)
    computes all integers $|x_0| < X = B^{2\alpha-1}/2$ for some $\alpha >
    1/2$ such that there exists $D > B^{\alpha}$ dividing
    both $A+x_0$ and $B$, or reports that no such
    $x_0$ exists. The algorithm runs in deterministic polynomial time
    in \(\log B\).
\end{proposition}
\begin{proof}
    It is enough to use Proposition \ref{HG_CF} for all possible $\alpha >
    1/2$, or equivalently test all convergents of $A/B$ (as noted above, since
    \(A/B\) is rational, its sequence of convergents is finite, and has
    polynomial length).
    Algorithm~\ref{alg:ACDCF} begins by computing these convergents.
\end{proof}

\begin{remark}
    The bound $X$ in Proposition~\ref{prop:CF}
    can be relaxed to $B^{2 \alpha-1}$
    (without the factor of \(1/2\)) 
    if we use intermediate convergents,
    but asymptotically this has no real importance.
\end{remark}

\begin{remark}\label{followup}
If we want {\em all} solutions $(x, D)$, then we have to include all
solutions coming from divisors of $D$, in the sense of Remark
\ref{sqf-A-B}---but finding the divisors of $D$ would imply resorting
to non deterministic and/or non-polynomial-time algorithms.
\end{remark}

\begin{algorithm}[hbt]
    \caption{Computing approximate common divisors using continued fractions.}
    \label{alg:ACDCF}
    \Function{\ACDCF{$A$, $B$}}{
        \Input{$A < B$}
        \Output{The set of solutions \((x,D)\) 
            to the \ACDP{} for \((A,B)\)
            (so $D \mid (A + x)$ and $D \mid B$)
            with $|x| < X := \frac{1}{2}B^{2 \alpha-1}$ 
            and $D > D_0 := B^{\alpha}$ 
            for some $\alpha > 1/2$.
        }
        \( (g_0/h_0,\ldots,g_n/h_n) \gets \)
        continued fraction convergents of $A/B$
        \;
        \(\mathcal{R} \gets \emptyset\)
        \;
        \For{\(i \gets 0\) \KwTo \(n\)}{ 
            \If{$h_i \mid B$ (and $h_i > 1$)}{
	            $D \gets B/h_i$\;
		        $x \gets D g_i - A$\;
                \(\mathcal{R} \gets \mathcal{R} \cup \{(x, D)\}\)
                \;
            }
        }
        \Return{\(\mathcal{R}\)}
    }
\end{algorithm}

\subsection{Computing approximate common divisors via lattice reduction}
\label{sec:ACDL}

Theorem~\ref{thm:ACD_L} enlarges the set of $\alpha$ for which we can find the
factorization of $N$. The proof of correctness can be found
in~\cite{Howgrave-Graham01}; optimal parameters are given in
Algorithm \ref{alg:ACDL}.
Remark \ref{followup} applies here too.

\begin{theorem}[Howgrave-Graham]
    \label{thm:ACD_L}
    Given integers \(A < B\),
    and \(\alpha\) in \((1/2,1)\) 
    and \(\beta\) in \((0, \alpha^2)\),
    Algorithm~\ref{alg:ACDL} (\ACDL)
    computes all $x$ such that there is some $D$ with \((x, D)\) a solution
    to the \ACDP{} for \((A,B)\) with $|x| < X := B^\beta\), $D > D_0 =
    B^\alpha$, in deterministic polynomial time
    in \(\log B\) and \(1/\epsilon\) where $\epsilon = \alpha ^2-\beta$.
\end{theorem}

\begin{algorithm}[hbt]
    \caption{Approximate common divisors using LLL}
    \label{alg:ACDL}
    \Function{\ACDL{$A$, $B$, $\alpha$, $\beta$}}{
        \Input{$A < B$ 
            and $\alpha \in (1/2, 1)$, $\beta \in (0, \alpha^2)$ 
        }
        \Output{The set of solutions \((x,D)\) to the \ACDP{} for \((A,B)\)
            (so $D \mid (A + x)$ and $D \mid B$)
            with $|x| < X:=B^{\beta}$ and $D > D_0:= B^{\alpha}$.
        }
        $h \gets \lceil \alpha (1-\alpha)/\epsilon\rceil - 1$
            where $\epsilon=\alpha^2-\beta > 0$
        \;
        $u \gets \lceil h \alpha\rceil$
        \;
        \(L \gets\) the \((h+1)\)-dimensional lattice of
        \(\tilde{p}_i\)-coefficients 
        defined in the proof of Theorem~\ref{thm:ACD_L}
        \;
        \((v_0,\ldots,v_h)
        \gets
        \ShortVector{L}\)
        \tcp*{Use LLL;
	    each $v_i$ is divisible by $X^i$}
        $P(Z) \gets \sum_{i=0}^h (v_i/X^i) Z^i$
        \;
        $\mathcal{X} \gets$ integer roots of $P(Z)$
        \;
        $\mathcal{R} \gets \emptyset$
        \;
        \For{$x \in \mathcal{X}$}{
            $D \gets \gcd(A+x, B)$
            \;
            \If{$1 < D < B$}{
               $\mathcal{R} \gets \mathcal{R} \cup \{(x, D)\}$
               \;
            }
        }
        \Return{\(\mathcal{R}\)}
        \;
    }
\end{algorithm}

%

\subsection{Algorithms to refine partial factorizations}
Many algorithms (including some given below) return nontrivial divisors of~\(N\),
rather than complete prime factorizations.
We can improve the quality of these partial factorizations
using some basic auxiliary algorithms, that all run in deterministic
polynomial time. The following two algorithms are taken from~\cite{BaDrSh93}.
        
\Refine takes a set of integers \(\{M_1,\ldots,M_k\}\),
and returns a set of pairs \((N_i,e_i)\)
with each \(N_i > 1\) and \(e_i > 0\),
and with the \(N_i\) all pairwise coprime,
such that \(\prod_i M_i =\prod_i N_i^{e_i}\).
This can be done by iterating the rewriting formula
\[
    M_1 M_2 = (M_1/d) (d^2) (M_2/d)
    \quad
    \text{where}
    \quad
    d = \gcd(M_1, M_2)
    \,.
\]
Faster algorithms for \Refine appear in~\cite{Bernstein05} and ~\cite{BeLePi07}.

\CleanDivisors takes an integer \(m\)
and a list of divisors \((d_1,\ldots,d_k)\) of \(m\),
and returns a set of pairs \((m_i,e_i)\)
such that \(m = \prod_i m_i^{e_i}\) where the $m_i$ are
pairwise coprime and such that each 
$d_i = \prod_j m_j^{e_{i, j}}$ for some $e_{i, j} \geq 0$.
This can be done by applying \Refine
to $\{d_1, m/d_1, \ldots, d_k, m/d_k\}$,
which yields \(\{(n_1,f_1),\ldots,(n_\ell,f_\ell)\}\)
such that $\prod_{i=1}^\ell n_i^{f_i} = m^k$ 
with the \(n_i\) pairwise coprime.
The $f_i$ are all multiples of~$k$, 
so the result 
is $\{(n_1,{f_1/k}), \ldots, (n_\ell,{f_\ell/k})\}$.

\section{
    Finding particular divisors of an integer
}
\label{sct-prep}

This section describes algorithms that find 
a large divisor \(D\) of $N$
if 
$(D-z)\mid M$ for an auxiliary integer $M$ and some small \(z\).
We use the simplest case, where \(D = p\) is prime and \(z = 1\)
(resp. $z = -1$)
for factoring with \(\Phi\) (resp. $\Sigma$) in~\S\ref{sct-Phi},
but we think that these more general results have independent interest.


\subsection{Factoring with unknown difference}

First, consider the search for divisors \(D\) of \(N\)
such that that \((D - z)\mid M\) 
where \(M\) is \emph{given}
and a small \(z \not= 0\) is \emph{unknown}.
For our needs (factoring $N$), the interesting case has $\gcd(N, M)=1$.
We can compute such \(D\) in deterministic polynomial time by
reduction to an \ACDP{} instance as follows.
Let \(y = M/(D-z)\),
so \(M=y (D-z)=yD-yz\);
computing the product \(yz\) 
leads to the divisor \(D\) by computing \(\gcd(N, M+yz)\) since
$\gcd(N, M)=1$. 
But \(x=yz\) is the solution of the modular equation
\(
    M+x \equiv 0 \pmod{D}
\),
and thus \((x, D)\) is a solution to \ACDP{} for $(A, B) = (M, N)$.
At this point, we finish using the results of~\S\ref{sct-tools}.

\begin{theorem}\label{key2}
    Let $1/2 < \alpha < 1$ be a real number.
    Let \(N\) and \(M\) be {\em coprime} integers and
    put $M = N^{\theta}$ with $1/2 < \alpha < \theta < 1$.
    Suppose there exists $D > N^{\alpha}$ such that \(D\mid N\)
    and \((D - z)\mid M\) 
    for a small \emph{unknown} integer \(z \not= 0\), with $|z| \leq
    N^{\varepsilon}$ such that $0 < \varepsilon < \alpha$.
    Then we can compute \(D\) in deterministic polynomial time in the
    following two cases:
    \begin{enumerate}
        \item
            $\varepsilon \leq 3 \alpha - 1 - \theta$;
        \item
            $\varepsilon < \alpha^2+\alpha - \theta$.
    \end{enumerate}
\end{theorem}
\begin{proof}
    Write $\beta = \log (M/(D-z))/\log N \approx \theta-\alpha$,
    so that $y = N^{\beta}$ 
    and $x = y z = N^{\theta+\varepsilon-\alpha}$.
    In Case~(1) we have $\theta+\varepsilon-\alpha \leq 2 \alpha - 1$,
    and Proposition~\ref{prop:CF} applies.
    In Case~(2),
    since \(|z| < N^{\alpha^2-\beta}\), we get $|x| \leq N^{\alpha^2}$,
    and we can compute \(x\) in deterministic polynomial time
    by Theorem~\ref{thm:ACD_L}.
\end{proof}

\subsection{Factoring with known difference}

Now we consider the opposite case:
finding \(D\mid N\) such that \((D - z) \mid M\)
where \(z\) is \emph{known}.
In our applications with $\Phi$ and $\Lambda$, we take \(z = 1\);
with $\Sigma$ we take $z=-1$.
In full generality, provided \(z\) is especially small,
we use Coppersmith's bivariate method from Theorem~\ref{cop}
to obtain the following result.

\begin{theorem}\label{key}
    Let \(N\), \(u\), and $M$ be integers
    with \(u \not= 0\) and \(|u| = N^{\varepsilon}\),
    and \(M = N^\theta\) with $0\leq \varepsilon < \theta < 1$. Fix $0
    \leq \varepsilon < \alpha < \theta$.
    Suppose there exists $D > N^{\alpha}$ such that \(D\mid N\)
    and \((D-u)\mid M\).
    Then Algorithm~\ref{alg:key} computes $D$ in deterministic
    polynomial time if 
    \begin{equation}\label{alpha}
        \alpha > \frac{1}{4}\, (1+\theta)
        \,.
    \end{equation}
\end{theorem}
\begin{proof}
    Rewrite the problem as
    $N = x_0 D$ and $M = y_0 (D-u)$, 
    so $M + u y_0 = y_0 D$. 
    Eliminating $D$, 
    we see that $(x_0,y_0)$ is a zero of 
    $f(x, y) = N y - x (M + u y) = -M x + N y - u x y .$
    If $D > N^{\alpha}$, then $x_0 <
    N^{1-\alpha}$ and $y_0 < N^{\theta-\alpha}$ are both small, 
    and we can use Theorem~\ref{cop}.  
    First, as in \cite{Coron04}, we let 
    $$f^*(x, y) := f(x, y+1) = N - (M+u) x + N y - u x y.$$
    Now \(f^*\) is irreducible, and linear in \(x\) and \(y\), so
    it meets the conditions of Theorem~\ref{cop}
    with \(\delta = 1\); and \(f^*(0,-1) = 0\).
    Assume $|x_0| < X$ and $|y_0| < Y$.
    The crucial bound is
    $$\mathcal{W} = \Vert f^*(x X, y Y)\Vert^* = \max (N, (M+u) X, N Y, X Y).$$
    Using $(X,Y) = (N^{1-\alpha},N^{\theta-\alpha})$
    gives
    $XY = N^{1+\theta-2 \alpha}$
    and
    $$\mathcal{W} = \max(N, N^{1+\theta-\alpha}, N^{1+\theta-2
    \alpha}) = N^{1+\theta-\alpha}.$$
    Ignoring small constants, we want 
    $$
        1+\theta-2 \alpha < \frac{2}{3}(1+\theta-\alpha)
    $$
    which implies
    $
        \alpha 
        > 
        \frac{1}{4} (1+\theta)
    $;
    the result follows.
\end{proof}
\begin{remark}
    A weaker but simpler result can be obtained using 
    Coron's algorithm, as in~\cite[\S 2]{Coron04}:
    if we use
    $f^*(x, y) = N - (M+u) x + N y - u x y$,
    then 
    $\alpha > (1+\theta)/3$ is enough to recover $D$.
\end{remark}
\begin{remark}
    If $u < 0$, then we have $\theta \geq \alpha$ which leads to $\theta > 1/3$.
\end{remark}

\begin{corollary}\label{cor:key}
    Using the notation of Theorem~\ref{key}:
    we can recover $D$
    in deterministic polynomial time provided $D > N^{1/2}$.
\end{corollary}
\begin{proof}
    It suffices to observe that $\alpha > 1/2 > (1+\theta)/4$ for $0 \leq \theta < 1$.
\end{proof}

\begin{algorithm}[hbt]
    \caption{Factoring with known difference}
    \label{alg:key}
    \Function{\FactoringWithKnownDifference{$N$, $M$, $u$, $\alpha$}}{
        \Input{Positive integers $N$ and $M$, an integer $u$,
            and a real number $\alpha$ such that 
            $(1+\theta)/4 < \alpha < \theta$ where
            $\theta := \log M / \log N$}
        \Output{$\{D > N^{\alpha} : D \mid N \text{ and } (D-u)\mid M\}$}
        $f^* \gets N - (M+u) x + N y - u x y$ in $\ZZ[x,y]$
        \;
        $\mathcal{R} \gets \BivariateCoppersmith(f^*, X, Y)$
        {where} \(X = N^{1-\alpha}\)
        {and} \(Y = N^{\theta-\alpha}\)
        \;
        $\mathcal{D} \gets \emptyset$
        \;
        \For{$(x, y)$ in $\mathcal{R} \setminus \{(0, -1)\}$}{
            $\mathcal{D} \gets \mathcal{D} \cup \{N/x\}$
        }
        \Return{$\mathcal{D}$}\;
    }
\end{algorithm}

\section{
    Factoring with the \texorpdfstring{$\Phi$}{Phi},
    \texorpdfstring{$\Lambda$}{Lambda}, and
    \texorpdfstring{$\Sigma$}{Sigma} oracles
}
\label{sct-Phi}

We now return to factoring with oracles.
We treat the closely-related problems
of factoring with \(\Phi\), \(\Lambda\), or $\Sigma$
simultaneously here,
before treating \(\ORDER\) in~\S\ref{sct-order}.
We consider odd \(N\),
since detecting and removing powers of \(2\) is easy.
Ordering the prime divisors of \(N\) by decreasing size,
we write
\[
    N = \prod_{i=1}^{\omega(N)} p_i^{e_i}
    \quad
    \text{with primes}
    \quad
    p_1 > p_2 > \cdots > p_{\omega(N)} > 2
    \,.
\]
To simplify the
exposition, the function associated with an oracle $\varpi$ will be
denoted by $\mathcal{F}(\varpi)$ 
(e.g., $\mathcal{F}(\Phi)=\varphi$).

\subsection{Reduction to the squarefree case}
\label{sct:sqf}

We begin by reducing to the case of squarefree \(N\):
that is, \(e_1 = \cdots = e_{\omega(N)} = 1\).
\begin{theorem}[Landau \cite{Landau88}]
    \label{th:Landau}
    Given \(N\) and $\varpi \in \{\Phi, \Lambda, \Sigma\}$,
    Landau's algorithm returns in deterministic polynomial time a list
    \((N_1,\ldots,N_r)\)
    such that \(N = N_1N_2^2\cdots N_r^r\),
    each \(N_i\) is squarefree or 1,
    and the \(N_i\) are pairwise coprime
    using $O(\omega(N))$ calls to $\varpi$.
\end{theorem}

\subsection{Reduction to the case \texorpdfstring{$\gcd(N, \mathcal{F}(\varpi)(N)) = 1$}{gcd(N,phi(N))=1}}

Suppose $N$ is squarefree. 
For $\varpi = \Phi$, $\Lambda$, or $\Sigma$,
if $\gcd(N, \mathcal{F}(\varpi)(N)) > 1$
then we obtain a nontrivial divisor $d$ of \(N\), and we can
combine the factorizations of $d$ and $N/d$ recursively.
Thus, we reduce to the problem of factoring
squarefree \(N\) where \(\gcd(N,\mathcal{F}(\varpi)(N)) = 1\).

\subsection{Products of two primes}
\label{varphik2}

It is well-known that
we can factor \(N = p_1p_2\) 
given \(\varphi(N)\),
as we recall in
Lemma~\ref{lemma:factorization-from-varphi}.
This immediately yields Algorithm~\ref{Phi2} (\FactorizationWithPhiTwo),
which factors a squarefree integer \(N\) with \(\omega(N) = 2\) given
$M = \varphi(N)$. Rephrased, this gives also that oracle $\Phi$ can
answer the decision problem of determining whether $\omega(N) = 2$.

\begin{lemma}
    \label{lemma:factorization-from-varphi}
    If \(N\) is a product of two distinct primes,
    then the two primes are 
    \[
        s/2 \pm \sqrt{(s/2)^2 - N}
        \quad
        \text{where}
        \quad
        s := N + 1 - \varphi(N)
        \ .
    \]
\end{lemma}
\begin{proof}
    If \(N = p_1 p_2\) with \(p_1\) and \(p_2\) prime,
    then \(\varphi(N) = (p_1-1)(p_2-1) = N - (p_1+p_2) + 1\);
    so \(s = p_1 + p_2\), 
    and \(p_1\) and \(p_2\) are the roots of the quadratic equation
    \(X^2 - sX + N\).
\end{proof}

\begin{algorithm}[hbt]
    \caption{Factoring a 2-factor integer using $M = \varphi(N)$ \label{Phi2}}
    \label{alg:factorization-with-phi-2}
    \Function{\FactorizationWithPhiTwo{$N$, $M$}}{
        \Input{$N$ and $M = \varphi(N)$, where $N$ is squarefree}
        \Output{%
            $\{p_1, p_2\}$ if $N$ is the product of two distinct primes,
            or $\emptyset$
        }
        $s \leftarrow N+1-M$
        \;
        \(\Delta \gets s^2 - 4N\)
        \tcp*{\(\Delta = \) discriminant of \(X^2 - sX + N\)}
        \If{\(\Delta\) is not square}{
            \Return{\(\emptyset\)}
        }
        \(p_1 \gets \frac{1}{2}(s + \sqrt{\Delta})\)
        \;
        \(p_2 \gets N/p_1\)
        \;
        \Return{$\{p_1, p_2\}$}
        \;
    }
\end{algorithm}

To convert Algorithm~\ref{alg:factorization-with-phi-2}
into an algorithm taking \(\lambda(N)\) instead of \(\varphi(N)\),
we use Lemma~\ref{lemma:lambda-to-varphi},
which shows that when \(\omega(N) = 2\),
we can efficiently compute \(\varphi(N)\) from \(\lambda(N)\).
Thus, any algorithm calling \(\Phi\) can be immediately
transformed into an algorithm making the same number of calls to~\(\Lambda\).
In particular, Algorithm~\ref{Phi2}
can be used with \(M = \lambda(N)\cdot\gcd(N-1,\lambda(N))\) 
instead of \(\varphi(N)\).

\begin{lemma}
    \label{lemma:lambda-to-varphi}
    If \(N = p_1p_2\) is a product of two distinct primes, 
    then \(\varphi(N) = \lambda(N)\cdot\gcd(N-1,\lambda(N))\).
\end{lemma}
\begin{proof}
    Suppose $N = p_1p_2$.
    Write 
    $g = \gcd(p_1-1, p_2-1)$;
    then 
    $p_1-1 = g q_1$ and $p_2-1 = g q_2$ with $\gcd(q_1, q_2) = 1$.
    Now
    $$\lambda(N) = (p_1-1)(p_2-1)/g = g q_1 q_2 \ , $$
    from which
    $\gcd(N-1,\lambda(N)) = g\cdot\gcd(gq_1q_2 + q_1 + q_2, q_1q_2)$,
    but
    $\gcd(g q_1 q_2 + q_1 + q_2, q_1 q_2) = 1$.
\end{proof}
Finally, for the oracle $\Sigma$, given $\sigma(N) = N+1+p_1+p_2$,
we immediately recover $p_1+p_2$ and then compute $p_1$ and $p_2$ as above.

\subsection{Products of more than two primes}
\label{varphikgt2}

Returning to the general squarefree case,
suppose
\[
    N = p_1\cdots p_k
    \quad
    \text{with primes}
    \quad p_1 > \cdots > p_k
    > 2
    \quad
    \text{and}
    \quad
    \omega(N) = k \ge 3
    \,.
\]
The relative sizes of the $p_i$ will be important in
what follows. We set
\[
    \alpha_i := \log_N p_i
    \ ,
    \quad
    \text{so}
    \quad
    p_i = N^{\alpha_i}
    \ . 
\]
Clearly $\sum_{i=1}^k \alpha_i = 1$
and \(1 > \alpha_1 > \cdots > \alpha_k > 0\);
so, in particular, $\alpha_1 > 1/k$ and $\alpha_k < 1/k$.

We first rephrase Corollary~\ref{cor:key} 
to show that {\em unbalanced} numbers (having a large prime factor) 
are easy to factor with $\varpi \in \{\Phi, \Lambda, \Sigma)$. In
contrast, {\em compact} \(N\) (with all prime factors $\leq N^{1/2}$)
are harder to factor. This gives us a result
already stated (in a simple form) as Theorem~\ref{thm11}.
\begin{theorem}\label{propcor:key}
  If $\omega(N) \geq 3$ and $\alpha_1 > 1/2$, then we can
  recover the divisor $D = p_1$ of $N$ in deterministic polynomial
  time in $\log(N)$ given $\mathcal{F}(\varpi)(N)$ for $\varpi \in
  \{\Phi, \Lambda, \Sigma\}$.
\end{theorem}
\begin{proof}
    We use $(D-1)\mid \varphi(N)$ (resp. $(D-1) \mid \lambda(N)$);
    the result follows directly from Corollary~\ref{cor:key}. The same
    holds for $\Sigma$ using $(D+1) \mid \sigma(N)$.
\end{proof}
\begin{remark}
When using $\lambda(N)$ in Theorem~\ref{propcor:key},
we can recover $p_1$ in deterministic polynomial time 
provided $\alpha_1 > (1+\theta)/4$,
where $\theta = \log \lambda(N)/\log N$.
When $\theta$ is significantly smaller than~\(1\), 
this gives a substantially lower bound on $\alpha_1$;
but finding a condition analogous to
Inequality~\eqref{eq:funky-condition} 
is not so easy in that case.
\end{remark}

The results of~\S\ref{sct-prep}
yield conditions on the \(\alpha_i\)
under which factors of \(N\)
can be computed with the algorithms of~\S\ref{sct-tools}.
Theorems~\ref{thm52}
and~\ref{th:ACDL-theoretical} show that we can factor \(N\) by
solving \ACDP{} instances if the \(p_i\) satisfy certain relative
size conditions. 
As a first step, Theorem~\ref{thm52} gives conditions for efficient
factoring using Algorithm~\ref{algo_factor} (\SplitCF),
which applies \ACDCF using \(\Phi\) or $\Sigma$. 

\begin{theorem}\label{thm52}
    Suppose $\omega(N) \geq 3$
    and there
    exists $1 \leq r < \omega(N)$ 
    such that
    \begin{equation}
        \label{eq:funky-condition}
        \alpha_r \geq 2 \sum_{i=r+1}^{\omega(N)} \alpha_i
        \,. 
    \end{equation}
    Then \ACDCF
    recovers the factor $D = \prod_{i=1}^r p_i$ 
    in deterministic polynomial time given \(\varphi(N)\) or~$\sigma(N)$.
\end{theorem}
\begin{proof}
    Write \(\alpha = \sum_{i=1}^r\alpha_i\).
    The hypothesis implies \(\alpha > 1/2\);
    otherwise 
    \(\alpha_r \geq 2 (1 - \alpha)\)
    and \(\alpha \le 1/2\),
    hence \(\alpha_r \ge 1\),
    which is impossible.
    Expanding the formula for \(\varphi(N)\)
    yields
    $\varphi(N) = D Q_1 - N/p_r + Q_2 $
    for some \(Q_1\) and \(Q_2\).
    If \(x = N/p_r - Q_2\),
    then \((x, D)\)
    is a solution to the \ACDP{}
    for \((A,B) = (\varphi(N),N)\)
    with \((M,X) = (N^\alpha,N^{2\alpha-1})\),
    and \ACDCF 
    will find \((x, D)\) 
    because \(\alpha > 1/2\).
    In this case
    $x \approx N/D = N^{1-\alpha}$,
    and the condition becomes 
    $1-\alpha_r \leq 2\alpha - 1$,
    which yields Inequality~\eqref{eq:funky-condition}.

    The same reasoning is valid for $\sigma(N)$, simply changing signs
    to get $\sigma(N) = D Q_1 + N/p_r + Q_2 $. (Strictly speaking, we
    should use $(A, B) = (\sigma(N)-N, N)$ to get $A < B$: see Remark
    \ref{AltB}.)
\end{proof}
\begin{algorithm}[hbt]
    \caption{Splitting an integer using ACDCF}
    \label{algo_factor}
    \Function{\SplitCF{$N$, $\varpi$}}{
        \Input{$N$ to be factored using oracle $\varpi \in \{\Phi, \Sigma\}$}
        \Output{$\emptyset$ or a set of pairs $(M_i,e_i)$,
            with the $M_i$ pairwise coprime
            and $N = \prod_i M_i^{e_i}$}
        %
        $M \stackrel{\varpi}{\longleftarrow} \mathcal{F}(\varpi)(N)$
        \;
        $sgn \gets \begin{cases}
            -1 & \text{if } \varpi = \Phi
            \\
            1 & \text{if } \varpi = \Sigma
        \end{cases}
        $
        \;
        \If(\tcp*[f]{$N$ is prime}){$M = N+sgn$}{
            \Return{$\{(N,1)\}$}
        }
        $\mathcal{A} \gets \ACDCF(M, N)$
        \;
        \For{$(x, D)$ in $\mathcal{A}$}{
            $\mathcal{D} \gets \mathcal{D} \cup \{D,N/D\}$
        }
        \If{$\mathcal{D} = \emptyset$}{
            \Return{$\emptyset$}
            \;
        }
        \Return{\CleanDivisors{$N$, $\mathcal{D}$}}
    }
\end{algorithm}

We can go further using \ACDL instead of \ACDCF.
Theorem~\ref{th:ACDL-theoretical} is the corresponding analogue
of Theorem~\ref{thm52}.
\begin{theorem}
    \label{th:ACDL-theoretical}
    If there exist
    \(\alpha\) in \((1/2,1)\) and \(\beta\) in \((0,\alpha^2)\)
    such that \(\alpha \le \sum_{i=1}^r \alpha_i\)
    and \(1-\alpha_r \le \beta\)
    for some \(1 \leq r < \omega(N)\),
    then we can 
    recover the divisors \(D = p_1\cdots p_r\) 
    and \(N/D = p_{r+1}\cdots p_{\omega(N)}\) of \(N\)
    in deterministic polynomial time
    given \(\alpha\) and \(\beta\), using
    \(\Phi\) or $\Sigma$. 
\end{theorem}
\begin{proof}
Write
$$\varphi(N) = [(p_1-1) (p_2-1) \cdots (p_r-1)] K
= D K - p_1 p_2 p_{r-1} p_{r+1} \cdots p_{\omega(N)} + E$$
where $E$ is negligible with respect to $N/p_r$. We obtain
$$\varphi(N) = D K - (N/p_r) + E.$$
Now, Theorem~\ref{thm:ACD_L} will use \ACDL 
    given \(A = \varphi(N)\), 
    \(B = N\),
    $\alpha \le \sum_{i=1}^r \alpha_i$,
    and $\beta \le 1 - \alpha_r$
    to find \((x, D)\)
    where \(D = p_1 p_2 \ldots p_r\)
    and \(x = N/p_r - E \approx N^{1-\alpha_r}\).

    The same conclusion holds for $\sigma(N) = D K + N/p_r + E'$.
\end{proof}

Theorem~\ref{th:ACDL-theoretical}
is difficult to apply directly,
because of the subtlety alluded to in~\S\ref{sec:ACDL}:
it is not enough to simply know that \(\alpha\) and \(\beta\)
satisfying the bounds \emph{exist},
because we need to use them as parameters to \ACDL.
On the other hand,
\ACDL does not need their \emph{exact} values
(indeed, if we knew the exact value for \(\beta = 1-\alpha_r\),
then we would already know the prime factor \(p_r = N^{\alpha_r}\)).
If we can guess that a suitable \(r\) exists,
then we can give a lower bound for \(\alpha_r\)
implying a lower bound for \(\alpha\)
and an upper bound for \(\beta\)
that allow us to apply \ACDL.
While the bounds may be far from the optimal values 
of \(\alpha\) and \(\beta\), 
thus yielding sub-optimal performance for \ACDL,
the solution is still polynomial time,
and it allows us to factor some integers that \ACDCF cannot.

\begin{definition}
    For each positive integer \(r\),
    we define a constant
    \begin{align*}
        \overline{\alpha}_r & := \frac{-1 + \sqrt{1 + 4r^2}}{2r^2}
        \,.
        \intertext{%
            The first few of these constants are
        }
        \overline{\alpha}_1 
        & = 
        (-1+\sqrt{5})/2 \approx 0.618
        \,,
        \\
        \overline{\alpha}_2 
        & = 
        (-1+\sqrt{17})/8 \approx 0.3904
        \,,
        \\
        \overline{\alpha}_3 
        & = (-1 + \sqrt{37})/{18} \approx 0.2824
        \,.
    \end{align*}
\end{definition}
    
\begin{lemma}
    \label{lemma:alpha-r}
    If \(\alpha_r > \overline{\alpha}_r\)
    for some \(0 < r < \omega(N)\),
    then \(r\),
    \(\alpha = r\overline{\alpha}_r\),
    and \(\beta = 1-\overline{\alpha}_r\)
    meet the conditions of Theorem~\ref{th:ACDL-theoretical}.
\end{lemma}
\begin{proof}
    Let \(\widetilde{\alpha} = \sum_{i=1}^r\alpha_i\)
    and \(\widetilde{\beta} = 1 - \alpha_r\);
    these are the ideal values for \(\alpha\) and \(\beta\)
    when applying Theorem~\ref{th:ACDL-theoretical}.
    Clearly \(\widetilde{\alpha} > r\alpha_r\).
    We can therefore use Theorem~\ref{th:ACDL-theoretical}
    with \(\alpha = rX\) and \(\beta = 1 - X\)
    for any \(X \le \alpha_r\)
    such that \(1 - X < (rX)^2\);
    that is, as long as \(X > \overline{\alpha}_r\).
    Moreover,
    \(1/2 < r\overline{a}_r < 1\)
    for all~\(r > 0\).
    Hence
    \((\alpha,\beta) = (r\overline{\alpha}_r,1-\overline{\alpha}_r)\)
    meets the conditions of the theorem for the given \(r\).
\end{proof}

We emphasize that Lemma~\ref{lemma:alpha-r} only gives a
\emph{sufficient} condition for suitable \(\alpha\) and \(\beta\),
but we can use it to turn the proof of
Theorem~\ref{th:ACDL-theoretical} into an effective algorithm.

\begin{theorem}
    \label{th:ACDL-practical}
    Fix an integer \(R > 1\).
    If there exists an \(0 < r < \min(R+1,\omega(N))\)
    for which \(\alpha_r \ge \overline{\alpha}_r\),
    then Algorithm~\ref{algo_ACDL} (\SplitLLL)
    recovers the divisor \(D = p_1\cdots p_r = N^\alpha\)
    of \(N\) in deterministic polynomial time 
    using \(\Phi\) or $\Sigma$.
\end{theorem}
\begin{proof}
    Algorithm~\ref{algo_ACDL}
    tries to factor \(N\)
    by calling \ACDL
    using increasing values of \(r\)
    (up to and including \(\min(R+1,\omega(N))\), 
    which in any case is trivially bounded by \(\log_2 N\),
    though much smaller values of \(R\) are more interesting),
    with the bounds for \(\alpha\) and \(\beta\)
    suggested by Lemma~\ref{lemma:alpha-r}.
    The result therefore follows
    from \(R\) serial applications of
    Theorem~\ref{th:ACDL-theoretical}.
\end{proof}

\begin{algorithm}[hbt]
    \caption{Splitting an integer using ACDL}
    \label{algo_ACDL}
    \Function{\SplitLLL{$N$, $\varpi$, $R$}}{
        \Input{$N$ to be factored using oracle $\varpi\in\{\Phi, \Sigma\}$,
        and a bound \(R > 1\) on putative \(r\)}
        \Output{$\emptyset$ or a set of pairs $(M_i,e_i)$,
            with the $M_i$ pairwise coprime
            and $N = \prod_i M_i^{e_i}$}
        %
        $M \stackrel{\varpi}{\longleftarrow} \mathcal{F}(\varpi)(N)$
        \;
        $\mathcal{D} \gets \emptyset$
        \;
        \For{$r \gets 1$ \KwTo $R$}{
            $\overline{\alpha}_r \gets (-1+\sqrt{1+4 r^2})/(2 r^2)$
            \;
            \(
                \mathcal{A} 
                \gets 
                \ACDL(M, N, r \overline{\alpha}_r, 1-\overline{\alpha}_r)
            \)
            \tcp*{use $(\alpha,\beta) = (r\overline{\alpha}_r,1-\overline{\alpha}_r)$}
            \For{\((x, D)\) in \(\mathcal{A}\)}{
                \(\mathcal{D} \gets \mathcal{D} \cup \{(D,N/D)\}\)
            }
        }
        \If{$\mathcal{D} = \emptyset$}{
            \Return{$\emptyset$}\;
        }
        \Return{\CleanDivisors{$N$, $\mathcal{D}$}}
    }
\end{algorithm}

\subsection{Products of exactly three primes}

We can say a little more for the special case 
of squarefree~\(N\) with \(\omega(N) = 3\).
The 
difficult part is in breaking $N$: once a non-trivial divisor is
found, we are left with a prime and a product of two primes that can
be easily factored recursively using the oracle.

Write
\[
    N = p_1p_2p_3
    \quad
    \text{where}
    \quad
    p_1 > p_2 > p_3
    \ .
\]
As usual, we set \(\alpha_i = \log_N p_i\);
by definition, \(1 > \alpha_1 > \alpha_2 > \alpha_3 > 0\),
and 
\(\alpha_3\) is completely determined by \((\alpha_1,\alpha_2)\)
because \(\alpha_1+\alpha_2+\alpha_3 = 1\).
Lemma~\ref{lemma:domain-of-validity}
defines the polygon in the \((\alpha_1,\alpha_2)\)-plane
corresponding to 
the domain of validity of the exponents for $\omega(N)=3$. 
\begin{lemma}
    \label{lemma:domain-of-validity}
    With \(N\) and \(\alpha_i = \log_N p_i\) defined as above,
    \((\alpha_1,\alpha_2)\)
    lies in the region of the \((\alpha_1,\alpha_2)\)-plane
    defined by the inequalities
    \begin{align*}
        0 < \alpha_2 & < \alpha_1
        \ ,
        &
        \alpha_1 + \alpha_2 & < 1
        \ ,
        &
        \alpha_1 & > 1/3
        \ ,
        &
        2\alpha_1 + 3\alpha_2 & > 3/2
        \ .
    \end{align*}
\end{lemma}
\begin{proof}
    The first three inequalities follow immediately from 
    the definition of the \(\alpha_i\).
    For the last,
    if $2 \alpha_1 + 3 \alpha_2 \leq 3/2$
    then $\alpha_2 \leq (3/2- 2 \alpha_1)/3$,
    whence
    $1-\alpha_1 = \alpha_2 + \alpha_3 < 2 \alpha_2 \leq 2/3(3/2-2 \alpha_1)$,
    so $\alpha_1/3 < 0$, 
    which is impossible.
\end{proof}

Figure~\ref{fig1} depicts the values of $(\alpha_1, \alpha_2)$ 
that our methods can tackle, shading in various regions of the polygon of
Lemma~\ref{lemma:domain-of-validity}. 
Each result applies only to the \emph{interior} of the corresponding
region, and does not apply to points on the boundary lines.
We can factor \(N\) using $\Phi$ (resp. $\Sigma$) with
\begin{itemize}
    \item
        Theorem~\ref{propcor:key} when \(\alpha_1 > 1/2\),
        so \((\alpha_1,\alpha_2)\) is in the diagonally shaded polygon;
    \item
        Theorem~\ref{thm52} with \(r = 2\)
        when 
        \(\alpha_2 \ge 2\alpha_3\),
	which translates as $2 \alpha_1 + 3 \alpha_2 \geq 2$,
        so
        $(\alpha_1, \alpha_2)$ is in the horizontally shaded polygon
        with vertices
        $(2/5, 2/5)$, $(1/2, 1/2)$, $(1/2, 1/3)$;
    \item
        Theorem~\ref{th:ACDL-practical}
        with \(r = 2\)
        when \(\alpha_2 > \overline{\alpha}_2\),
        so \((\alpha_1,\alpha_2)\)
        is in the tiny black triangle with vertices 
        $(\overline{\alpha}_2, \overline{\alpha}_2)$, 
        $(2/5, 2/5)$, 
        $(1-3 \overline{\alpha}_2/2 = 0.415, \overline{\alpha}_2)$.
\end{itemize}
The grey polygon
with vertices
$(1/3, 5/18)$, $(1/3, 1/3)$,
$(\overline{\alpha}_2, \overline{\alpha}_2)$,
$(1-3 \overline{\alpha}_2/2, \overline{\alpha}_2)$,
$(1/2, 1/3)$, and $(1/2, 1/6)$
is the zone where 
we cannot prove deterministic polynomial-time factorization.

A necessary condition to apply Theorem~\ref{th:ACDL-theoretical} in
our case for $r=2$ ($r=1$ being uninteresting)
is
\begin{align*}
    1-\alpha_2 & < (\alpha_1+\alpha_2)^2 \,, 
    \shortintertext{or}
    \alpha_2 & > f(\alpha_1) 
    := \frac{-(2 \alpha_1+1) + \sqrt{4 \alpha_1+5}}{2}
    \,.
\end{align*}
The function \(f\) is decreasing on $[0, 1]$ and is smaller than $\alpha_2$
for $\alpha_1 \ge (\sqrt{17}-1)/8 \approx 0.3904$; note that $f(1/2) =
(\sqrt{7}-1)/2 \approx 0.3229$. This is the dash-dotted line, which
corresponds to a sharp limit on using this theorem.

\iffalse 
\begin{figure}[hbt]
\begin{center}
\begin{tikzpicture}[xscale=10,yscale=8,>=latex]
\draw[->] (-0.25, 0) -- (1.25, 0); \node at (1.2, -0.1) {$\alpha_1$};
\draw[->] (0, -0.25) -- (0, 0.7); \node at (-0.075, 0.65) {$\alpha_2$};
\draw[fill,black!35] (0.5, 0.166) -- (0.333, 0.277) --
(0.333, 0.333) -- (0.4, 0.4) -- (0.5, 0.333) -- (0.5, 0.166);
\draw[dashed] (0, 0) -- (0.333, 0.333);
\draw[dashed] (0.333, 0.277) -- (0, 0.5);
\node at (-0.05, 0.277) {$5/18$};
\draw[dashed] (0, 0.277) -- (0.333, 0.277);
\draw[fill,white] (0.666, 0.0555) -- (0.666, 0.333)
-- (1, 0) -- (0.75, 0);
\draw[pattern=north east lines] (0.666, 0.0555) -- (0.666, 0.333)
-- (1, 0) -- (0.75, 0);
\node[anchor=north] (1/3) at (0.333, -0.15) {$1/3\;\;$};
\draw[dashed] (0.333, -0.15) -- (0.333, 0.333);
\node[anchor=north] (2/3) at (0.7, -0.15) {$2/3$};
\draw[dashed] (0.618, -0.24) -- (0.618, 0.088);
\node[anchor=north] (overalpha1) at (0.618, -0.25) {$\overline{\alpha}_1$};
\draw[dashed] (0.666, -0.15) -- (0.666, 0.0555); 
\draw[fill,pattern=dots] (0.618, 0.088) -- (0.618, 0.254) -- (0.666,
0.222) -- (0.666, 0.0555);
\draw[fill,pattern=crosshatch] (0.5, 0.166) -- (0.5, 0.333) -- (0.618,
0.254) -- (0.618, 0.088) -- (0.5, 0.166);
\node[inner sep=0pt] (1/2x0) at (0.5, 0) {$\times$};
\node[inner sep=0pt,anchor=north] (1/2x0lab) at (0.525, -0.05) {$1/2$};
\draw[dashed] (1/2x0) -- (0.5, 0.166);
\draw[fill,pattern=horizontal lines] (0.666, 0.222) -- (0.4, 0.4) -- (0.5, 0.5) --
(0.666, 0.333) -- (0.666, 0.222);
\node[inner sep=0pt] (2/5x0) at (0.4, 0) {$\times$};
\node[inner sep=0pt,anchor=north] (2/5x0lab) at (0.4, -0.05) {$2/5$};
\node[inner sep=0pt,anchor=north] (1x0lab) at (1, -0.05) {$1$};
\node[inner sep=0pt,anchor=east] (0,2/5) at (0, 0.4) {$2/5\;\;$};
\node[inner sep=0pt,anchor=east] (0,1/2) at (0, 0.5) {$1/2\;\;$};
\draw[dashed] (0, 0.5) -- (0.5, 0.5);
\node[inner sep=0pt] (3/4x0) at (0.75, 0) {$\times$};
\node[inner sep=0pt,anchor=north] (3/4x0lab) at (0.75, -0.05) {$\;\;3/4$};
\draw[-] (0.333, 0.277) -- (0.75, 0);
\draw[-] (0.333, 0.277) -- (0.333, 0.333);
\draw[dashed] (0, 0.4) -- (0.4, 0.4); \draw[dashed] (0.4,0) -- (0.4, 0.4);
\end{tikzpicture}
\caption{Cases covered by our results when $\omega(N) = 3$.\label{fig1}}
\end{center}
\end{figure}
\else 
\begin{figure}[hbt]
\begin{center}
\begin{tikzpicture}[xscale=8,yscale=8,>=latex]
\draw[->] (-0.2, 0) -- (1.15, 0); 
  \node[anchor=north] at (1.1, -0.02) {$\alpha_1$};
\draw[->] (0, -0.075) -- (0, 0.6); 
  \node at (-0.05, 0.57) {$\alpha_2$};
\draw[fill,black!35] (0.5, 0.166) -- (0.333, 0.277) --
(0.333, 0.333) -- (0.4, 0.4) -- (0.5, 0.333) -- (0.5, 0.166);
\draw[dashed] (0, 0) -- (0.333, 0.333);
\draw[dashed] (0.333, 0.277) -- (0, 0.5);
\node at (-0.05, 0.277) {$5/18$};
\draw[dashed] (0, 0.277) -- (0.333, 0.277);
\draw[pattern=north east lines] (1, 0) -- (0.75, 0) -- (0.5, 0.166) --
(0.5, 0.5);
\node[inner sep=0pt,anchor=north] (1/3) at (0.333, -0.02) {$1/3$};
\draw[dashed] (0.333, 0) -- (0.333, 0.333);
\draw (0.5, 0.5) -- (1, 0);
\node[inner sep=0pt] (1/2x0) at (0.5, 0) {}; 
\node[inner sep=0pt,anchor=north] (1/2x0lab) at (0.5, -0.02) {$1/2$};
\draw[dashed] (1/2x0) -- (0.5, 0.5);
\draw[fill,pattern=horizontal lines] (0.4, 0.4) -- (0.5, 0.5) --
(0.5, 0.333) -- (0.4, 0.4);
\node[inner sep=0pt] (2/5x0) at (0.4, 0) {}; 
\node[inner sep=0pt,anchor=north] (2/5x0lab) at (0.4, -0.02) {$2/5$};
\node[inner sep=0pt,anchor=north] (1x0lab) at (1, -0.02) {$1$};
\node[inner sep=0pt,anchor=east] (0,2/5) at (0, 0.4) {$2/5\;\;$};
\node[inner sep=0pt,anchor=east] (0,1/2) at (0, 0.5) {$1/2\;\;$};
\draw[dashed] (0, 0.5) -- (0.5, 0.5);
\node[inner sep=0pt] (3/4x0) at (0.75, 0) {}; 
\node[inner sep=0pt,anchor=north] (3/4x0lab) at (0.75, -0.02) {$3/4$};
\draw[-] (0.333, 0.277) -- (0.75, 0);
\draw[-] (0.333, 0.277) -- (0.333, 0.333);
\draw[dashed] (0, 0.4) -- (0.4, 0.4); \draw[dashed] (0.4,0) -- (0.4,
0.4);
\draw[fill,black] (0.4, 0.4) -- (0.415, 0.3904) -- (0.3904, 0.3904);
\draw[-] (0.4, 0.4) -- (0.415, 0.3904) -- (0.3904, 0.3904);
\draw[dash dot] (0.3904, 0.3904) -- (0.5, 0.3229);
\end{tikzpicture}
\caption{Cases covered by our results when $\omega(N) = 3$.\label{fig1}}
\end{center}
\end{figure}
\fi

\subsection{Numerical examples}

We use Algorithms~\ref{alg:ACDCF} (\ACDCF) and~\ref{alg:ACDL} (\ACDL) 
to factor various $N$ given \(\varphi(N)\) or \(\lambda(N)\).
The algorithms succeed when the divisors of $N$
satisfy the required properties.

We start with a numerical example for each sub-region in Figure~\ref{fig1}).

\begin{example}[\SplitCF with $\Phi$]
    Consider an attempt to factor
    \begin{align*}
        N & = 14300000000000000000000000045617
        \intertext{
            using \SplitCF.  The oracle \(\Phi\) tells us that
        }
        \varphi(N) & = 12000000000000000000000000038160 \,.
        \intertext{
            Applying \ACDCF with \(A = \varphi(N)\) and \(B = N\)
            reveals that $N$ has a divisor
        }
        D & = 100000000000000000000000000319 \,,
    \end{align*}
    which turns out to be prime; 
    the cofactor is $143 = 13\cdot11$.
    In this case, $\alpha_1 = 0.93082\ldots > 1/2$.
\end{example}
\begin{example}
    Let us factor
    \begin{align*}
        N & = 215453441884154813899608536725827949716396214692299863
        \intertext{
            using \SplitCF again.  The oracle \(\Phi\) tells us that
        }
        \varphi(N) & = 215453439729720123839763043257007943520875035231793568 \,.
        \intertext{
            Applying \ACDCF with \(A = \varphi(N)\) and \(B = N\)
            reveals that $N$ has a divisor
        }
        D & = 2154434690059745273387380265425792937208898747 \,,
    \end{align*}
which has two prime factors. In this case $(\alpha_1, \alpha_2,
    \alpha_3) = (0.45, 0.4, 0.15)$, a point in the horizontally shaded
    part.
\end{example}
\begin{example}
    Let us factor
    \begin{align*}
        N &= 14300000027170000072930000138567
        \intertext{
            with \SplitCF.  The oracle \(\Phi\) gives
        }
        \varphi(N) &= 12000000021600000060000000108000 \,,
        \intertext{
            and then \ACDCF with \((A,B) = (\varphi(N),N)\)
            finds a divisor 
        }
        D &= 100000000190000000510000000969
    \end{align*}
    with two prime factors $10000000019$ and $10000000000000000051$, 
    and the cofactor $N/D = 143$. We have $(\alpha_1, \alpha_2,
    \alpha_3, \alpha_4) = (0.610, 0.321, 0.0358, 0.033)$ and
    Inequality~\eqref{eq:funky-condition} is satisfied for $r = 2$.
\end{example}

\begin{example}[\SplitLLL with $\Phi$]
    Let us factor
    \begin{align*}
        N &= 5872731058374808693660010068837
        \intertext{
            with \SplitLLL.  The oracle \(\Phi\) gives
        }
        \varphi(N) &= 5872725180353869744164863505600 \,,
        \intertext{
            and then \ACDL with \((A,B) = (\varphi(N),N)\)
            finds a divisor 
        }
        D &= 5878015394214207265992137\,, \\
        x &= 5544735287880571100\,,
    \end{align*}
and $D = 1425101895589\cdot 4124628149333$, the third factor being 
$999101$. We have $(\alpha_1, \alpha_2,
    \alpha_3) = (0.410, 0.395, 0.195)$ which is in the tiny triangle.
\end{example}

\begin{example}[Factoring with $\Sigma$]
    Let us factor
    \begin{align*}
        N 
        &= 
        2682776312933147882428349713219285333356964534315603933540
        \\
        & \quad \ 90095217359233
        \intertext{using \SplitLLL.  The oracle $\Sigma$ tells us that}
        \sigma(N)
        &=
        2682803140502033115557437331732180804000067423327043796969
        \\
        & \quad \ 97748284763200
        \,.
        \intertext{
            Trying \(r = 2\),
            and calling \ACDL
            with 
            \(
                (A,B)
                =
                (\sigma(N)-N, N)
            \)
            and
            \(
                (\alpha,\beta)
                = 
                (2\overline{\alpha}_2, 1-\overline{\alpha}_2)
            \)
            (implying lattice parameters $(h,u)=(25,20)$),
            we find a solution
        }
      D &= 610540229658532834519888426420070208770724882201228981991
        \,,
        \\
        x &= -23576265633281739760211511675892594424044680
        \,.
    \end{align*}
\end{example}

\begin{example}[Factoring with $\Lambda$]
    We apply \SplitCF to
    \begin{align*}
        N &= 14300000027170000072930000138567
        \,,
        \intertext{
            for which the oracle \(\Lambda\)
            tells us that
        }
        \lambda(N) &= 100000000180000000500000000900
        \,.
        \intertext{
            \ACDCF reveals a divisor
        }
        D &= 100000000190000000510000000969
    \end{align*}
    with two factors, 
    $10000000019$ and $10000000000000000051$ 
    (which we find recursively using \FactorizationWithPhiTwo),
    and a cofactor $N/D = 143 = 11\cdot 13$.
    We see that $\lambda(N)/\varphi(N)=1/120$ (so 
    $\lambda(N)$ is close to $\varphi(N)$, and the method may work), and 
    \[
        (\alpha_1,\alpha_2,\alpha_3,\alpha_4)
        =
        ( 0.60984, 0.32097, 0.035754, 0.033425 )
        \,.
    \]
\end{example}

\section{
    Other oracles
}
\label{sec:other}

Before concluding, we briefly survey some logical extensions to other
oracles that do not yield useful results.

\subsection{Using the factorization of \texorpdfstring{$\varphi(N)$}{phi(N)} or \texorpdfstring{$\sigma(N)$}{sigma(N)}}
\label{ssct:factphi}

Every odd prime $p\mid N$ is necessarily of the form $\delta+1$ 
for some even $\delta\mid \varphi(N)$, so we can compute all prime
factors $p$ of $N$ from the factors of $\varphi(N)$. Unfortunately
this does not lead to polynomial-time algorithms, since the number
of divisors of $\varphi(N)$ can be large, as shown in ~\cite{LuPo07},
and the same should hold for $\lambda(N)$ as well.

If $\omega(N)=k$ (for squarefree $N$, say), then the smallest prime factor
of $N$ has $p_k < N^{1/k}$, so we might content ourselves with
finding $p_k$ by
enumerating divisors of $\varphi(N)$ less than
$\varphi(N)^{\alpha} = N^{1/k}$. But this is not enough
to get polynomial time, since this number can be lower bounded by $C
d(\varphi(N))^{-C' \alpha \log \alpha}$ for positive constants $C$ and $C'$
(see~\cite{Wolke72}, studying a function introduced by P.~Erd\H{o}s in \cite{Erdos52}).

We anticipate the same properties for $\sigma(N)$.

\subsection{Factoring with the order oracle}
\label{sct-order}

We now consider factoring using the order oracle \(\ORDER\), whose
quantum counterpart is the core of Shor's algorithm. As explained in
\cite{GrLaSmMo16}, when $\lambda(N)$ has very few divisors, having the
order of an element is enough to factor $N$.
It is doubtful that we can find an algorithm for all integers, since
$\lambda(N)$ may have a lot of divisors that cannot help factoring $N$.

Suppose we have the factorization of the order.
As in~\S\ref{ssct:factphi},
we might consider a modified \(\ORDER\)
that yields not only the order \(r\) of \(a\) modulo \(N\),
but also the factorization of \(r\).
Algorithm~\ref{algo_fact_ord}
shows a straightforward way to make use of this additional information.
If $N$ is not squarefree, 
then it is possible that $\gcd(r, N) \not= 1$,
which gives us an easy factor of \(N\)
(hence the check in Line~\ref{gcdrN}).
Algorithm \ref{algo_fact_ord} fails, returning $\emptyset$,
if $a$ has order $r$ modulo every prime factor $p_i$ of~$N$, 
or if $r \mid p_i-1$ for all $i$, which implies that all
divisors of $N$ are congruent to \(1\pmod{r}\).
Then, if $r > N^{1/4+\varepsilon}$, we can conclude in deterministic
polynomial time using Theorem~\ref{theorem:co-ho-na}.
Another approach for large $r$ is given in~\cite{Zralek10}.

\subsection{Combining different oracles}
\label{sct-combining}

In another direction, having $\varphi(N)$ \emph{and} $\sigma(N)$ yields
the factorization of squarefree $N$ with three factors by finding the
integer roots of the polynomial $(X-p_1)(X-p_2)(X-p_3) = X^3 +
(N-(\sigma(N)+\varphi(N))/2) X^2 + ((\sigma(N)-\varphi(N))/2-1) X -
N$, extending the result of \FactorizationWithPhiTwo.

\section{Conclusions}

We have shown a range of partial results concerning the relationships between
several elementary number theoretic functions
and the integer factorization problem.
In each case, we have used ideas coming from lattice reduction to improve
what was known, while falling short of the goal of completely proving the
sufficiency of these oracles for efficiently factoring all numbers. 

As we saw in~\S\ref{sec:other},
adding more information does not pay: The complete factorizations of
oracle values (or given $\varphi(N)$ {\em and} $\lambda(N)$, or even
given their prime factorizations) \emph{still} does not help factoring
all $N$. 
These results may be surprising, but they show the fundamental
difficulty of factoring.

\subsection*{Acknowledgements.} We thank J.~Shallit
for sending us a copy of \cite{Long81}, and B.~{\'Z}ra\l{}ek for
sending us a copy of his work~\cite{Zralek19}. We are grateful to
W.~George for bringing \cite{Chow15} to our attention. J.-L.~Nicolas
and G.~Tenenbaum were kind enough to send us results related to
Wolke's work. All algorithms
were programmed and tested in \textsc{Magma}, and some computations were
done in \textsc{Maple}.

\iffalse
\bibliographystyle{plain}
\bibliography{oracles}
\else
\def\noopsort#1{}\ifx\bibfrench\undefined\def\biling#1#2{#1}\else\def\biling#1#2{#2}\fi\def\Inpreparation{\biling{In
  preparation}{en
  pr{\'e}paration}}\def\Preprint{\biling{Preprint}{pr{\'e}version}}\def\Draft{\biling{Draft}{Manuscrit}}\def\Toappear{\biling{To
  appear}{\`A para\^\i tre}}\def\Inpress{\biling{In press}{Sous
  presse}}\def\Seealso{\biling{See also}{Voir
  {\'e}galement}}\def\Editor{\biling{Ed.}{R{\'e}d.}}

\fi

\newpage
\appendix
\section{Splitting and factoring integers}
\label{sec:appendix}


In this appendix, we give several algorithms to split an integer
(i.e., finding a non-trivial divisor) or to factor it completely,
summarising our work.
\begin{itemize}
    \item \FactorWithEvenPower (Alg.~\ref{algo:ord-factor}) is a
        primitive for factoring using an element of even order;
    \item \SplitWithOracleRandom (Alg.~\ref{algo:sq-oracle-rnd}) is
        a randomized version of factoring with oracles, together with
        \FactorWithOracleRandom (Alg.~\ref{algo:fact-oracle}). A version
        for $\Sigma$ can be found in~\cite{BaMiSh86}.
    \item \FactorWithOracle is the main function for factoring using
        an oracle (Alg. \ref{algo:sq-oracle}), and its ancillary function
        \FWO (Alg.~\ref{algo:FWO});
    \item \FactorWithFactoredOrder (Alg.~\ref{algo_fact_ord}) 
        uses the factorization order oracle.
\end{itemize}

\begin{algorithm}[hbt]
    \caption{Factoring with an even power}
    \label{algo:ord-factor}
    \SetKwInOut{Input}{input}\SetKwInOut{Output}{output}
    \Function{\FactorWithEvenPower{$N$, $a$, $k$}}{
        \Input{Integers \((N,a,k)\) with $k$ even, \(\gcd(a,N) = 1\),
        and $a^k\equiv 1\bmod N$}
    \Output{a non-trivial divisor $1 < d < N$ or {\tt failure}}
    \BlankLine
    Compute \((s,t)\) such that \(k = 2^s\cdot t\) with $s > 0$, $t$ odd\;
    \tcp{by hypothesis, $a^{2^s\cdot t} \equiv 1\bmod N$}
        \(b \gets a^t\bmod N\) \;
	\If{\(b \neq 1\)}{
	    find the smallest $s'$, $1 \leq s' \leq s$ such that $b^{2^{s'}}
        \equiv 1\bmod N$\;
        \(c \gets b^{2^{s'-1}}\bmod N\)
	    \tcp*{$c$ is a square root of 1}
	\If{$c \neq -1$}{
        \Return{$\gcd(c-1, N)$}
	}
	}
    \Return {\tt failure}\;
}
\end{algorithm}
If $k = \ord_N(a)$, then surely, we cannot have $b = 1$.
When $k = \varphi(N)$ (resp. $\lambda(N)$), the probability that $b = 1$ 
is bounded by $1/2^s \leq 1/2^{\omega(N)}$. There are
$2^{\omega(N)}$ square roots of $1$ (by the Chinese Remainder Theorem), including
$\pm 1$. There are $2^{\omega(N)}-1$ possible values for $c$ and only
one is trivial. So $c \neq -1$ with probability $\geq
1-1/2^{\omega(N)-1} \geq 1/2$.

If we use the order oracle \(\ORDER\),
then we may need to try several random values of~$a$ until 
we find one with even order~\(r\). And again this happens with
probability $\leq 1/2^s$.

\begin{algorithm}[hbt]
    \caption{Splitting an integer with an oracle using randomness}
    \label{algo:sq-oracle-rnd}
    \SetKwInOut{Input}{input}\SetKwInOut{Output}{output}
    \Function{{SplitWithOracleRandom}($N$, $\varpi$)}{
    \Input{An integer \(N\), an oracle $\varpi \in \{\Phi, \Lambda, \ORDER\}$}
    \Output{$N$ if $N$ is prime, otherwise a non-trivial divisor $1 < d < N$}
    \BlankLine
    \If{$\varpi \in \{\Phi, \Lambda\}$}{
        $k \stackrel{\varpi}{\longleftarrow} \mathcal{F}(\varpi)(N)$\;
        \If(\tcp*[f]{$N$ is prime}){$k = N-1$}{
	        \Return $N$\;
        }
    }
    \While{true}{
        choose a random \(a \in [2, N-2]\) \;
        \(g \gets \gcd(a,N)\) \;
        \If{\(g \not=1\)}{
            \Return{\(g\)} \;
        }
	\If{$\varpi = \ORDER$}{
	    $k \stackrel{\varpi}{\longleftarrow} \ord_N(a)$\;
        \If(\tcp*[f]{$N$ is prime}){$k = N-1$}{
            \Return $N$\;
        }
	}
	\tcp{by hypothesis, $a^k \equiv 1\bmod N$}
	\If{$k$ is even}{
		$res\gets$ \FactorWithEvenPower{$N$, $a$, $k$}\;
		\If{$res \neq $ {\tt failure}}{
			  \Return{$res$}\;
		}
	}
    }
    }
\end{algorithm}

\begin{algorithm}[hbt]
    \caption{Factoring an integer with an oracle and randomness}
    \label{algo:fact-oracle}
    \SetKwInOut{Input}{input}\SetKwInOut{Output}{output}
    \Function{\FactorWithOracleRandom{$N$, $\varpi$}}{
    \Input{Integer \(N\), oracle $\varpi \in \{\Phi, \Lambda, \ORDER\}$}
    \Output{A set $\{(p_1, e_1), \ldots, (p_r, e_r)\}$
    s.t. $N = \prod_{i=1}^r p_i^{e_i}$ with all \(p_i\) prime}
    \BlankLine
    $d \gets$ \SplitWithOracleRandom{$N$, $\varpi$} \;
    \If(\tcp*[f]{$N$ is prime}){$d = N$}{
	    \Return $\{(N, 1)\}$\;
    }
    \Else{
        $\mathcal{L}_0 \gets$ \Refine{$\{d, N/d\}$} \;
	    \tcp{$\mathcal{L}_0 = \{(M_1, e_1), \ldots, (M_s, e_s)\},
            \gcd(M_i, M_j)=1 \text{ for } i \neq j$}
	    $\mathcal{L} \gets \emptyset$
        \;
	    \For{$(M, e) \in \mathcal{L}_0$}{
	        $\mathcal{L}_1 \gets$ \FactorWithOracle{$M$, $\varpi$} \;
            \tcp{$\mathcal{L}_1 = \{(p_1, f_1), \ldots, (p_u, f_u)\}$,
                $p_i$ prime}
            \For(\tcp*[f]{since all $M_i$ are coprime, $p$ is not in $\mathcal{L}$}){$(p, f) \in \mathcal{L}_1$}{
	            $\mathcal{L} \gets \mathcal{L} \cup \{(p, e f)\}$\;
	        }
        }
	    \Return{$\mathcal{L}$}\;
    }
}
\end{algorithm}

\begin{algorithm}[hbt]
    \caption{Factoring an integer with an oracle}
    \label{algo:sq-oracle}
    \SetKwInOut{Input}{input}\SetKwInOut{Output}{output}
    \Function{\FactorWithOracle{$N$, $\varpi$}}{
    \Input{A squarefree integer \(N\), an oracle $\varpi \in \{\Phi, \Sigma\}$}
    \Output{Sets $\mathcal{P}$ and $\mathcal{C}$ (possibly empty)
    containing prime and composite divisors of $N$, respectively}
    \BlankLine
	\uIf{$\varpi = \Phi$}{
		$sgn \gets -1$\;
	}
	\Else{
		$sgn \gets +1$\;
	}
    $M \stackrel{\varpi}{\longleftarrow} \mathcal{F}(\varpi)(N)$\;
    \If(\tcp*[f]{$N$ is prime}){$M = N+sgn$}{
       \Return $(\{N\}, \emptyset)$\;
    }
    $\mathcal{P} \gets$ \FactorizationWithPhiTwo{$N$} \;
    \If{$\mathcal{P} \neq \emptyset$}{
        \Return{$(\mathcal{P}, \emptyset)$}\;
    }
    $g \gets \gcd(N, M)$\;
    \If{$g \neq 1$}{
        \Return{\FWO{$N$, $\{g\}$, $\varpi$}} \;
    }
    $\mathcal{D} \gets$ \FactoringWithKnownDifference{$N$, $M$, $1$, $0.5$} \;
    \If(\tcp*[f]{we have found some $D > N_1^{1/2}$}){$\mathcal{D} \neq \emptyset$}{
        \Return{\FWO{$N$, $\mathcal{D}$, $\varpi$}} \;
    }
    \tcp{Theorem \ref{thm52} can be used?}
    $\mathcal{D} \gets$ \SplitCF{$N$, $\varpi$} \;
    \If{$\mathcal{D} \neq \emptyset$}{
        \Return{\FWO{$N$, $\mathcal{D}$, $\varpi$}} \;
    }
    $\mathcal{D} \gets$ \SplitLLL{$N$, $\varpi$, $\lfloor \log_2 N\rfloor$} \;
    \If{$\mathcal{D} \neq \emptyset$}{
        \Return{\FWO{$N$, $\mathcal{D}$, $\varpi$}} \;
    }
    \Return{$(\emptyset, \{N\})$}\;
}
\end{algorithm}

\begin{algorithm}[hbt]
    \caption{Ancillary function for Algorithm~\ref{algo:sq-oracle}}
    \label{algo:FWO}
    \SetKwInOut{Input}{input}\SetKwInOut{Output}{output}
    \Function{\FWO{$N$, $\mathcal{D}$, $\varpi$}}{
    \Input{Squarefree integer \(N\), an oracle $\varpi \in \{\Phi,
    \Sigma\}$, and a set $\mathcal{D}$ of non-trivial divisors of $N$}
    \Output{Two sets (possibly empty) $\mathcal{P}$ and $\mathcal{C}$
    where the former (resp. the latter) contains prime
    (resp. composite) divisors of $N$}
    \BlankLine
    $\mathcal{D} \gets \cup_{d \in \mathcal{D}} \{d, N/d\}$
    \tcp*{force $d$ and $N/d$ to be in the set}
    $\mathcal{D} \gets$ \CleanDivisors{$\mathcal{D}$}
    \;
    $(\mathcal{P},\mathcal{C}) \gets (\emptyset,\emptyset)$
    \;
    \For{$d \in \mathcal{D}$}{
    	$(\mathcal{P}_d, \mathcal{C}_d) \gets$ \FactorWithOracle{$d$, $\varpi$} \;
        $
            (\mathcal{P},\mathcal{C}) 
            \gets
            (\mathcal{P} \cup \mathcal{P}_d, \mathcal{C} \cup \mathcal{C}_d)
        $\;
    }
    \Return{$(\mathcal{P}, \mathcal{C})$}\;
}
\end{algorithm}

\begin{algorithm}[hbt]
    \caption{Factoring with factorization of order}\label{algo_fact_ord}
    \Function{\FactorWithFactoredOrder{$N$, $a$}}{
        \Input{$N$, $a$}
        \Output{A set of pairs \((M_i,e_i)\) with the \(M_i\) pairwise
        coprime and \(\prod_i M_i^{e_i} = N\)}
        \BlankLine
        \(
            \{(\ell_1,e_1),\ldots,(\ell_u,e_u)\}
	        \gets 
            \Factorization{$\ORDER(a,N)$}
        \)
        \;
        \(r \gets \Pi_{i=1}^u\ell_i^{e_i}\)
        \;
        \If(\tcp*[f]{$N$ is prime}){$r = N-1$}{
	        \Return{$\{(N, 1)\}$}\;
	    }
        $g \gets \gcd(r, N)$
        \;
        \If{$g \ne 1$}{ \label{gcdrN}
            $\mathcal{L} \gets \emptyset$
            \;
            \For{$i\gets 1$ \KwTo $u$}{
                $v \gets \nu_{\ell_i}(N)$ 
                \tcp*{maximal power of $\ell_i$ dividing $N$}
                \If{$v > 0$}{
                    $\mathcal{L} \gets \mathcal{L} \cup \{(\ell_i, v)\}$
                    \;
                    $N \gets N/\ell_i^v$
                    \;
                }
            }
            \tcp{$N > 1$ since $r < N$}
            \Return{$\mathcal{L} \cup {}$\FactorWithFactoredOrder{$N$, $a$}}
        }
        \(\mathcal{M} \gets \emptyset\)
        \;
        \For{$i \gets 1$ \KwTo $u$}{
            $b \gets a^{r/\ell_i} \bmod N$
            \;
            $g \gets \gcd(b-1, N)$
            \;
            \If{$1 < g < N$}{
                $\mathcal{M} \gets \mathcal{M} \cup \{g, N/g\}$
            }
        }
        \Return{\CleanDivisors{$N$, $\mathcal{M}$}} 
        \;
    }
\end{algorithm}

\end{document}